\documentclass[11pt]{amsart}

\usepackage{amssymb}

\usepackage{amssymb,latexsym,amsmath,amsfonts}
\usepackage{eucal}
\usepackage{color}
\usepackage{ mathrsfs}
\usepackage{dsfont}

\usepackage{ulem}

\DeclareSymbolFont{SY}{U}{psy}{m}{n}
\DeclareMathSymbol{\emptyset}{\mathord}{SY}{'306}

\theoremstyle{plain}

\newtheorem{thm}{Theorem}[section]
\newtheorem{cor}[thm]{Corollary}
\newtheorem{lem}[thm]{Lemma}
\newtheorem{prop}[thm]{Proposition}
\theoremstyle{definition}
\newtheorem{defn}[thm]{Definition}
\newtheorem{rem}[thm]{Remark}

\numberwithin{equation}{section}

\newcommand{\bSB}{\boldsymbol{B}}

\def\W{With respect to}

\def\beq{\begin{eqnarray}}
\def\eeq{\end{eqnarray}}
\def\beqa{\begin{eqnarray*}}
\def\eeqa{\end{eqnarray*}}

\begin{document}

\title[Curvature inequalities and the Wallach set]{Curvature inequalities for operators\\
in the Cowen-Douglas class and localization\\ of the Wallach set}

\author{Gadadhar Misra and Avijit Pal}

\address{department of mathematics,
Indian Institute of Science, Bangalore - 560 012, India}
\email{gm@math.iisc.ernet.in} \email{apal@math.iisc.ernet.in}

\keywords{Cowen-Douglas class, Bergman kernel, Bergman metric, infinite divisibility, curvature inequality, contractive and completely contractive homomorphisms, Wallach set}

\subjclass[2010]{Primary 47A13; Secondary 32M15, 32A36 and 46E22}

\thanks{The work of G.Misra was supported, in part, through the J C Bose National Fellowship and UGC-SAP IV.
The work of A. Pal was supported, in part, through the UGC-NET
Research Fellowship and the IFCAM Research Fellowship. The results of this paper form part of his PhD thesis at the Indian Institute of Science.}

\begin{abstract} For any bounded domain $\Omega$ in $\mathbb C^m,$ let ${\mathrm B}_1(\Omega)$ denote the Cowen-Douglas class of commuting $m$-tuples of bounded linear operators.
For an $m$-tuple $\boldsymbol T$ in the Cowen-Douglas class ${\mathrm B}_1(\Omega),$ let $N_{\boldsymbol T}(w)$ denote
the restriction of $\boldsymbol T$ to the subspace ${\cap_{i,j=1}^m\ker(T_i-w_iI)(T_j-w_jI)}.$
This commuting $m$-tuple $N_{\boldsymbol T}(w)$ of $m+1$ dimensional operators induces a homomorphism $\rho_{_{\!N_{\boldsymbol T}(w)}}$ of the polynomial ring $P[z_1, \cdots, z_m],$ namely, $\rho_{_{\!N_{\boldsymbol T}(w)}}(p) = p\big (N_{\boldsymbol T}(w) \big ),\, p\in P[z_1, \cdots, z_m].$ We study the contractivity and complete contractivity of the homomorphism $\rho_{_{\!N_{\boldsymbol T}(w)}}.$ Starting from the
homomorphism $\rho_{_{\!N_{\boldsymbol T}(w)}},$ we construct a natural class of
homomorphisms  $\rho_{_{\!{\boldsymbol N}^{(\lambda)}(w)}}, \lambda>0,$ and relate the properties of
$\rho_{_{\!{\boldsymbol N}^{(\lambda)}(w)}}$ to those of  $\rho_{_{\!N_{\boldsymbol T}(w)}}.$ Explicit examples arising
from the multiplication operators on the Bergman space of $\Omega$
are investigated in detail. Finally, it is shown that contractive
properties of $\rho_{_{\!N_{\boldsymbol T}(w)}}$ is equivalent to an  inequality  for the curvature of the Cowen-Douglas bundle $E_{\boldsymbol T}$.
\end{abstract}

\maketitle
\section{Introduction}
We recall the  definition of the well known class of operators
$\mathrm B_n(\Omega)$ which was introduced in the foundational
paper of Cowen and Douglas (cf. \cite{cowen}). An alternative
point of view was discussed in the paper of Curto and Salinas (cf.
\cite{curto}).
\begin{defn}
The class $\mathrm B_n(\Omega)$ consists of $m$-tuples of commuting
bounded operators  $\boldsymbol T=(T_1, T_2, \ldots, T_m)$ on a Hilbert
space $\mathcal H$ satisfying the following conditions:
\begin{itemize}

\item for $w=(w_1,\dots , w_m)\in\Omega,$ the dimension of the
joint kernel $\bigcap_{k=1}^{m}\ker(T_{k}-w_{k}I)$ is $n$,

\item for $w\in\Omega$ and $h\in \mathcal H,$ the operator $D_{\boldsymbol
T-wI} :\mathcal H \rightarrow {\mathcal H \oplus \ldots
\oplus \mathcal H},$ defined by the rule: $$D_{\boldsymbol
T-wI}h= \big ((T_{1}-w_1I)h, \ldots ,(T_{m}-w_mI)h\big )$$
has closed range,

\item the closed
linear span of $\{\bigcap_{k=1}^{m}\ker(T_{k}-w_{k}I):w\in\Omega\}$  is $\mathcal H.$
\end{itemize}
\end{defn}
The commuting $m$-tuple $N_{\boldsymbol T}(w)$ is obtained by restricting $T_i$ to the $(m+1)$
dimensional subspace $\mathcal N(w):= \cap_{i,j=1}^m {\ker} (T_i -
w_iI)(T_j -w_jI).$ These commuting $m$-tuples of finite
dimensional operators are included, for instance, in
the (generalized) class of examples due to Parrott (cf. \cite{BB, gmisra}).
The study of the
contractivity and complete contractivity of the homomorphisms
induced by these localization operators leads to interesting
problems in geometry of finite dimensional Banach spaces (cf.
\cite{BB,vern}).

For operators in the the Cowen-Douglas
class $\mathrm B_1(\Omega),$ there exists a holomorphic choice $\gamma(w)$ of
eigenvector at $w.$ The $(1,1)$ - form
$$ \mathsf K(w) = - \sum_{i=1}^m \frac{\partial^2}{\partial w_i \bar{\partial}w_j} \log \|\gamma(w)\|^2 dw_i \wedge d\bar{w}_j$$
is then seen to be a complete unitary invariant for these
operators. It will be useful for us to work with the matrix of 
co-efficients  of the $(1,1)$ - form defining the curvature
$\mathsf K$, namely,
$$\mathcal K_{\boldsymbol T}(w):= - \left(\!\!\left (\frac{\partial^2}{\partial
w_i\partial\bar{w}_j
}\log\|\gamma(w)\|^2\right)\!\!\right)_{i,j=1}^{m}.$$
The subspace $\mathcal N(w)$ is easily seen to be
spanned by the vectors $\{\gamma(w), (\partial_1 \gamma)(w),
\ldots , (\partial _m \gamma)(w)\}.$ The localization $N_{\boldsymbol
T}(w)$ represented with respect to the orthonormal basis obtained
via the Gram-Schmidt process from these vectors, takes the form:
$$
\Big ( \begin{smallmatrix} w_1 & \boldsymbol \alpha_1(w) \\ 0 & w_1\!I_m
\end{smallmatrix}\Big ), \ldots  ,\Big ( \begin{smallmatrix} w_m &
\boldsymbol \alpha_m(w) \\ 0 & w_m\!I_m \end{smallmatrix}\Big ).
$$
Cowen and Douglas, for $m=2$,  had shown that the curvature
$\mathcal K_{\boldsymbol T}(w)$ appears directly in the localization of the
operator $\boldsymbol T.$ Indeed, for all $m\geq 1,$ we show that the matrix of inner
products $ \big (\!\! \big ( \langle \boldsymbol \alpha_i(w) , \boldsymbol \alpha_j(w)
\rangle \big ) \!\!\big ) $ equals $- \mathcal K_{\boldsymbol T}(w)^{-1}.$

As is well-known, without loss of
generality, a Cowen-Douglas operator $\boldsymbol T$ may be thought of
as the adjoint of the commuting tuple  of multiplication operators
by the coordinate functions on a functional Hilbert space. The
holomorphic section $\gamma$, in this representation, becomes
anti-holomorphic, namely  $K_w(\cdot),$ where $K$ is the
reproducing kernel. One might ask what properties of the
multiplication operators (or the adjoint) are determined by the
local operators. For instance, let $M$ be the multiplication (by the coordinate function) operator on a Hilbert space consisting of  holomorphic functions
 on the unit disc $\mathbb D$ and possessing a reproducing kernel $K.$  The contractivity of the homomoprhism, induced by the operator $M,$ of the disc algebra (this is the same
 as requiring $\|M\| \leq 1,$ thanks to the von Neumann inequality) ensures
contractivity of the homomorphisms $\rho_{_{\!N_{M^*}(w)}}$ induced by
the local operators $N_{M^*}(w),$ $w\in \mathbb D,$ which is itself equivalent to an inequality for the curvature $\mathcal K_{M^*}$.
An example is given in \cite{shibu} showing that the contractivity of the homomorphisms
$\rho_{_{\!N_{M^*}(w)}}$ need not imply $\|M\| \leq
1$ and that the
converse is valid only after imposing some additional conditions. In
this paper, for an arbitrary bounded symmetric domain, we
construct such examples by exploiting  properties typical of the
Bergman kernel on these  domains.

In another direction, for any positive definite kernel $K$ and a positive real number $\lambda,$ the function $K^\lambda$ obtained by polarizing the real analytic function
$K(w,w)^\lambda,$ defines a Hermitian form, which is not necessarily positive definite. The determination of the Wallach set
$$\mathcal W_\Omega:=\{\lambda: K^\lambda(z,w)\mbox{\rm ~is positive definite~}\}$$
for this kernel is an important problem. The Wallach set was  first defined for the Bergman kernel of a bounded symmetric domain. Except in that case (cf. \cite{FK}), very little is known about the Wallach set in general. Here, for all $(0 < )\lambda$ and a fixed but arbitrary $w$ in $\Omega,$ we show that $K^\lambda(z,w)$ is positive definite when restricted to the $m+1$ dimensional subspace  
$$\mathcal {\mathcal N}^{(\lambda)}(w):= \vee\{ K^\lambda(\cdot, w), \bar{\partial}_1 K^\lambda(\cdot, w), \ldots  \bar{\partial}_m K^\lambda(\cdot, w)\}.$$   
This means
that for an arbitrary choice of complex numbers $\alpha_0, \ldots
, \alpha_m,$
$$ \sum_{i,j=0}^m\alpha_i \bar{\alpha}_j \big (\partial_i \bar{\partial}_j K^\lambda\big )(w,w) > 0$$
for a fixed but arbitrary $w\in \Omega$ and all $\lambda > 0.$ Here $\partial_0$ is set to be the scalar $1.$

For an $m$-tuple of operators $\boldsymbol T$ in $\mathcal {\mathrm
B}_n(\Omega)$, Cowen and Douglas establish the existence of a
non-vanishing holomorphic map $\gamma:\Omega_0\rightarrow\mathcal H$
with $\gamma(w)$ in $\bigcap_{k=1}^{m}\ker(T_{k}-w_{k}I)$, $w$ in
some open  subset $\Omega_0$ of $\Omega.$ We fix such an open set
and call it $\Omega$. The map $\gamma$ defines a holomorphic
Hermitian vector bundle, say $E_{\boldsymbol T}$, on $\Omega$.
They show that the equivalence class of the vector bundle
$E_{\boldsymbol T}$ determines the unitary equivalence class of the operator $\boldsymbol T$ and
conversely. The determination of the
 equivalence class of the operator $\boldsymbol T$ in $\mathrm B_1(\Omega)$ then is  particularly simple since
 the curvature $\mathsf K(w)$ of the line bundle $E_{\boldsymbol T}$
is a complete invariant. We reproduce the
well-known proof of this fact for the sake of completeness.

Suppose that $E$ is a holomorphic Hermitian line bundle over a
bounded domain $\Omega\subseteq \mathbb C^m$. Pick a holomorphic
frame  $\gamma$ for the line bundle $E$ and let $\Gamma(w)=
\langle{\gamma_w , {\gamma_w}\rangle}$ be the Hermitian metric.
The curvature $(1,1)$ form $\mathsf K(w)\equiv 0$ on an open
subset $\Omega_0 \subseteq \Omega$ if and only if $\log \Gamma$ is
harmonic on $\Omega_0$.  Let $F$ be a second line bundle over the
same domain $\Omega$ with the metric $\Lambda$ with respect to a
holomorphic frame $\eta$.  Suppose that the two curvatures
$\mathsf K_E$ and $\mathsf K_F$ are equal on the open subset
$\Omega_0$. It then follows that $u=\log (\Gamma/\Lambda)$ is
harmonic on this open subset.  Thus there exists a harmonic
conjugate $v$ of $u$ on $\Omega_0$, which we assume without loss
of generality to be simply connected. For $w\in\Omega_0$, define
$\tilde{\eta}_w = e^{(u(w)+iv(w))/2} \eta_w$. Then clearly,
$\tilde{\eta}_w$ is a new holomorphic frame for $F$. Consequently,
we have the metric $\tilde{\Lambda}(w) = \langle{\tilde{\eta}_w,
\tilde{\eta}_w \rangle}$ for $F$ and  we see that
\begin{align*}
\tilde{\Lambda}(w) &= \langle{\tilde{\eta}_w, \tilde{\eta}_w\rangle}\\
&= \langle{e^{(u(w)+iv(w))/2} {\eta}_w, e^{(u(w)+iv(w))/2}{\eta}_w \rangle}\\
&= e^{u(w)}\langle{{\eta}_w, {\eta}_w \rangle}\\
&= \Gamma(w).
\end{align*}
This calculation shows that the map $U:{\tilde{\eta}}_w \mapsto \gamma_w$
defines an isometric holomorphic bundle map between $E$ and $F$.
As shown in  \cite[Theorem 1]{douglas}, the map 
\begin{equation}\label{CD unitary}
U\Big ( \sum_{|I| \leq n} \alpha_I (\bar{\partial}^I
\eta)({w_0})\Big ) = \sum_{|I| \leq n} \alpha_I (\bar{\partial}^I
\gamma)(w_0), \,\, \alpha_I \in \mathbb C,
\end{equation}
where $w_0$ is a fixed point in $\Omega$ and $I$ is a multi-index
of length $n$, is well-defined, extends to a unitary operator on
the Hilbert space spanned by the vectors $(\bar{\partial}^I
\eta)({w_0})$ and intertwines  the two $m$-tuples of operators in
$\mathrm B_1(\Omega)$ corresponding to the vector bundles $E$ and
$F$.

It is natural to ask what other properties of $\boldsymbol T$ are
directly reflected in the curvature $\mathsf K.$ One such property
that we explore here is the contractivity and complete
contractivity of the homomorphism induced by the $m$-tuple
$\boldsymbol T$ via the map $\rho_{\boldsymbol T}: f\to
f(\boldsymbol T)$, $f\in \mathcal O(\Omega),$ where $\mathcal
O(\Omega)$ is the set of all holomorphic function in a 
neighborhood of $\overline{\Omega}.$


We recall the curvature inequality from Misra and Sastry {cf.
(\cite[Theorem 5.2]{GM}}) and produce a large family of examples to
show that the ``curvature inequality'' does not imply
contractivity of the homomorphism $\rho_{\boldsymbol T}$.

\section{Localization of Cowen-Douglas operators}
Fix an operator $\boldsymbol T$ in $\mathrm B_1(\Omega)$ and
let $N_{\boldsymbol T}(w)$ be the $m$-tuple of operators $(N_1(w),
\ldots, N_m(w)),$ where $N_i(w)= (T_{i}-w_{i}I)|_{\mathcal N(w)},$
$i=1,\ldots, m.$ Clearly, $N_i(w)N_j(w)=0$ for all $1\leq i,j \leq
m.$ Hence the commuting $m$-tuple $N(w)$ is jointly nilpotent.
This $m$-tuple of nilpotent operators $N_{\boldsymbol T}(w),$ since $(T_i-w_iI)\gamma(w) =0$ and
$(T_i-w_iI) (\partial_j\gamma)(w) = \delta_{ij} \gamma(w)$ for $1
\leq i,j \leq m$, has
the matrix representation $N_k(w) = \left ( \begin{smallmatrix}
0 & e_k \\
 0 & \boldsymbol 0\\
\end{smallmatrix}\right ),$  where $e_k$ is the vector $(0,\ldots , 1, \ldots ,0)$ with the $1$ in the $k${\tt th} slot, $k=1,\ldots , m.$   
Representing $N_k(w)$ with respect to an orthonormal basis in
$\mathcal N(w)$, it is possible to read off the curvature of
$\boldsymbol T$ at $w$ using the relationship:
\begin{equation}\label{curvform}
\big (-\mathcal K_{\boldsymbol T}(w)^{\rm t}\big )^{-1}=\big (
\!\! \big ( {\rm tr}\big ( N_{k}(w)\overline{N_{j}(w)}^{\rm t}\big
)\, \big )\!\!\big )_{k,j=1}^m = A(w)^{\rm
t}\overline{A(w)},\end{equation} where the $k${\tt th}-column of
$A(w)$ is the vector $\boldsymbol \alpha_k$ (depending on $w$)
which appears in the matrix representation of $N_k(w)$ with
respect to any choice of an orthonormal basis in $\mathcal N(w)$.

This formula is established for a pair of operators in $\mathrm
B_1(\Omega)$ (cf. \cite[Theorem 7]{douglas}).  However, we will
verify it for an $m$-tuple $\boldsymbol T$ in $\mathrm
B_1(\Omega)$ for any $m \geq1.$

Fix $w_0$ in $\Omega$. We may assume without loss of generality
that $\|\gamma(w_0)\|=1$. The function $\langle \gamma(w),
\gamma(w_0)\rangle$ is invertible in some neighborhood of $w_0$.
Then setting $\hat{\gamma}(w):=  \langle \gamma(w),
\gamma(w_0)\rangle^{-1} \gamma(w)$, we see that
$$\langle \partial_k \hat{\gamma}(w_0), \gamma(w_0) \rangle = 0, \,\, k=1,2,\ldots, m.$$
Thus $\hat{\gamma}$ is another holomorphic section of $E$.  The
norms of the two sections $\gamma$ and $\hat{\gamma}$ differ by
the absolute square of a holomorphic function, that is
$\tfrac{\|\hat{\gamma}(w)\|}{\|\gamma(w)\|} = |\langle \gamma(w),
\gamma(w_0)\rangle|$. Hence the curvature is independent of the
choice of the holomorphic section.
 Therefore, without loss of generality, we will prove the claim assuming,
for a fixed but arbitrary $w_0$ in $\Omega$, that
\begin{enumerate}
\item[{(i)}] $\|\gamma(w_0)\|=1$, \item[{(ii)}]
$\gamma(w_0)$ is orthogonal to $(\partial_k\gamma)(w_0)$,
$k=1,2,\ldots , m$.
\end{enumerate}

Let $G$ be the Grammian corresponding to the $m+1$ dimensional
space spanned by the vectors $$\{\gamma(w_0),
(\partial_1\gamma)(w_0),  \ldots , (\partial_m \gamma)(w_0)\}.$$
This is just the space $\mathcal N(w_0)$. Let $v, w$ be any two
vectors in $\mathcal N(w_0)$. Find $\boldsymbol c=(c_0,\ldots,
c_m), \boldsymbol d=(d_0, \ldots, d_m)$ in $\mathbb C^{m+1}$ such
that $v=\sum_{i=0}^{m}c_i {\partial}_i\gamma(w_0)$ and
$w=\sum_{j=0}^{m}d_j{\partial}_j \gamma(w_0).$  Here
$(\partial_0\gamma)(w_0) = \gamma(w_0)$. We have
\begin{align*}\langle v, w\rangle&=\langle\sum_{i=0}^{m}c_i{\partial}_i\gamma(w_0),
\sum_{j=0}^{m}d_j {\partial}_j \gamma(w_0)\rangle\\
&= \langle G^{\rm t}(w_0)\boldsymbol c, \boldsymbol d\rangle_{\mathbb C^{m+1}}\\
&= \langle (G^{\rm t})^{\frac{1}{2}}(w_0)\boldsymbol c, (G^{\rm
t})^{\frac{1}{2}}(w_0)\boldsymbol d\rangle_{\mathbb C^{m+1}}.
\end{align*}
Let $\{e_i\}_{i=0}^{m}$ be the standard orthonormal basis for
$\mathbb C^{m+1}$. Also, let $(G^{\rm
t})^{-\frac{1}{2}}(w_0)e_i:=\boldsymbol \alpha_i(w_0)$, where
$\boldsymbol \alpha_i(j)(w_0) = \alpha_{j i}(w_0)$, $i=0,1,\ldots,
m$. We see that the vectors  $\varepsilon_i:=\sum_{j=0}^m
\alpha_{ji} (\partial_j \gamma)(w_0)$, $i=0,1, \ldots ,m,$ form an
orthonormal basis in $\mathcal N(w_0)$:  \begin{align*}\langle
\varepsilon_i, \varepsilon_l\rangle &=\big\langle
\sum_{j=0}^{m}\alpha_{j i}{\partial}_j \gamma(w_0),
\sum_{p=0}^{m}\alpha_{p l}{\partial}_p \gamma(w_0)\big\rangle\\&=
\langle(G^{\rm t})^{-\frac{1}{2}}(w_0)\boldsymbol \alpha_i, (G^{\rm
t})^{-\frac{1}{2}}(w_0)\boldsymbol
\alpha_l\rangle\\&=\delta_{il},\end{align*} where
$\delta_{il}$ is the Kronecker delta. Since $N_k\big
(\,(\partial_j\gamma)(w_0) \,\big )= \gamma(w_0)$ for $j=k$ and
$0$ otherwise, we have $N_k(\varepsilon_i) = \Big (
\begin{smallmatrix} 0 & \boldsymbol \alpha_k^{\rm t}\\ 0 & 0
\end{smallmatrix} \Big )$. Hence
\begin{align*}
{\rm tr}\big ( N_i(w_0) N_j^*(w_0) \big ) &=  \boldsymbol{\alpha_i}(w_0)^{\rm t} \overline{\boldsymbol \alpha}_j(w_0)\\
&= \big ( (G^{\rm t})^{-\frac{1}{2}}(w_0)e_i\big )^{\rm t}\overline{\big ( (G^{\rm t})^{-\frac{1}{2}}(w_0)e_j\big )}\\
&= \langle {G}^{-\frac{1}{2}}(w_0)e_i , {G}^{-\frac{1}{2}}e_j(w_0)
\rangle = {(G^{\rm t})}^{-1}(w_0)_{ij}.
\end{align*}
Since the curvature, computed with respect to the holomorphic
section $\gamma$ satisfying the conditions {(i)} and
{(ii)}, is of the form
\begin{align*}
-\mathcal K_{\boldsymbol T}(w_0)_{ij} &= \frac{\partial^2}{\partial w_i \bar{\partial} w_j}\log \|\gamma(w)\|^2_{|w=w_0}\\
&= \Big ( \frac{\|\gamma(w)\|^2 \big ( \frac{\partial^2
\gamma}{\partial w_i \partial\bar{w}_j} \big )(w) - \big
(\frac{\partial \gamma}{\partial w_i}\big )(w)  \big (
\frac{\partial \gamma }{\partial \bar{w}_j}\big ) (w)}
{\|\gamma(w)\|^4}\Big )_{|w=w_0}\\
&= \big ( \frac{\partial^2 \gamma}{\partial w_i \partial
\bar{w}_j} \big )(w_0)=G(w_0)_{ij}, \end{align*} we have verified
the claim \eqref{curvform}.

The following theorem was proved for $m=2$ in ({cf. \cite[Theorem
7]{douglas}}). However, for any natural number $m$, the proof is
evident from the preceding discussion.
\begin{thm}
Two $m$-tuples of operators $\boldsymbol T$ and
$\tilde{\boldsymbol T}$ in $\mathrm B_1(\Omega)$ are unitarily
equivalent if and only if $N_k(w)$ and $\tilde{N}_k(w)$ are
simultaneously unitarily equivalent for $w$ in some open subset of
$\Omega$.
\end{thm}
\begin{proof}
Let us fix an arbitrary point $w$ in $\Omega$. In what follows,
the dependence on this $w$ is implicit. Suppose that there exists
a unitary operator $U:\mathcal N\rightarrow \widetilde{\mathcal
N}$ such that $UN_i=\tilde{N_i}U$, $i= 1,\ldots,m.$
For $1 \leq i, j \leq m,$ we have
\begin{align*}{\rm tr}\big ( \tilde{N_i}
\tilde{N_j}^* \big )&={\rm tr}\big (\big(U N_iU^*\big)\big(U
N_jU^*\big)^* \big )\\&={\rm tr}\big ( UN_i N_j^* U^*\big
)\\&={\rm tr}\big ( N_i N_j^* U^*U\big )\\&={\rm tr}\big ( N_i
N_j^* \big ).
\end{align*}
Thus the curvature of the operators $\boldsymbol T$ and $
{\tilde{\boldsymbol T}}$ coincide making  them unitarily
equivalent proving the Theorem in one direction. In the other
direction, observe that if the operators $\boldsymbol T$ and $
{\tilde{\boldsymbol T}}$ are unitarily equivalent then the unitary
$U$ given in \eqref{CD unitary} evidently maps $\mathcal N$ to
$\tilde{\mathcal N}$. Thus the restriction of $U$ to the subspace
$\mathcal N$ intertwines $N_k$ and $\tilde{N}_k$ simultaneously
for $k=1,\cdots ,m$.
\end{proof}

As is well-known (cf. \cite{curto} and \cite{cowen}), the $m$-tuple
$\boldsymbol T$ in $\mathrm B_1(\Omega)$ can be represented as the
adjoint of the $m$-tuple of multiplications $\boldsymbol M$ by the co-ordinate
functions on a Hilbert space $\mathcal H$ of holomorphic functions
defined on $\Omega^* = \{\bar{w}\in \mathbb C^m: w\in \Omega\}$
possessing a reproducing kernel $K:\Omega^*\times \Omega^* \to
\mathbb C$ which is holomorphic in the first variable and
anti-holomorphic in the second.

In this representation, if we set $\gamma(w) = K(\cdot, \bar{w})$,
$w\in \Omega$, then we obtain a natural non-vanishing
``holomorphic" map into the Hilbert space $\mathcal H$ defined on
$\Omega$.

The localization $N_{\boldsymbol T}(w)$ obtained from the commuting tuple of
operators $\boldsymbol T$
 defines a homomorphism $\rho_{_{\!N_{\boldsymbol T}(w)}}$ on the algebra $\mathcal O(\Omega)$ of functions,
each holomorphic in some neighborhood of the closed set $\bar{\Omega},$ by the rule
\begin{equation} \label{homN}
\rho_{_{\!N_{\boldsymbol T}(w)}}(f)=\left(\begin{matrix}f(w) & \nabla f(w) A(w)^{\rm t}\\
0 & f(w)I_m\end{matrix}\right),\,\, f\in \mathcal O(\Omega).
\end{equation}
We recall from  \cite[Theorem 5.2]{GM} that the
contractivity of the homomorphism
implies the curvature inequality $\|\big(\mathcal K_{\boldsymbol
T}(w)^{\rm t}\big)^{-1}\| \leq 1$. Here $\mathcal K_{\boldsymbol
T}(w)$ is thought of as a linear transformation from the normed
linear space $(\mathbb C^m, C_{\Omega,w})^*$ to the normed linear
space $(\mathbb C^m, C_{\Omega,w}),$ where $C_{\Omega, w}$ is the
Carath\'{e}odory metric of $\Omega$ at $w.$ The operator norm is
computed accordingly with respect to these norms.
\subsection{Infinite divisibility}
Let $K$ be a positive definite kernel defined on the domain
$\Omega$ and let $\lambda > 0$ be arbitrary. Since $K^\lambda$ is
a real analytic function defined on $\Omega$, it admits a power
series representation of the form
$$
K^\lambda(w,w) = \sum_{I,\!J} a_{I,J}(\lambda) (w-w_0)^I
\overline{(w-w_0)}^J
$$
in a small neighborhood of a fixed but arbitrary $w_0\in \Omega$.
The polarization $K^\lambda(z,w)$ is the function represented by
the power series
$$
K^\lambda(z,w) = \sum_{I,\!J} a_{I\,J}(\lambda) (z-w_0)^I
\overline{(w-w_0)}^J,\,\,w_0\in \Omega.
$$
It follows that the polarization $K^\lambda(z,w)$ of the function
$K(w,w)^\lambda$ defines a Hermitian kernel, that is,
$K^\lambda(z,w) = \overline{K(w,z)^\lambda}$. Schur's Lemma (cf.
\cite{bhatia})
  ensures the positive definiteness of $K^\lambda$ whenever $\lambda$ is a natural number. However,
  it is not necessary that $K^\lambda$ must be positive definite for all real $\lambda > 0$.
  Indeed a positive definite kernel $K$ with the property that $K^\lambda$ is positive definite for all
  $\lambda >0$ is called infinitely divisible and plays an important role in studying curvature inequalities
  (cf. \cite[Theorem 3.3]{shibu}).

Although, $K^\lambda$ need not be positive definite for all
$\lambda >0$, in general, a related question raised here is
relevant to the study of localization of the Cowen-Douglas
operators.

Let $w_0$ in $\Omega$ be fixed but arbitrary. Also, fix a $\lambda
>0$. Define the mutual inner product  of the vectors
$$
\{(\bar{\boldsymbol{\partial}}^I K^\lambda)(\cdot, w_0): I= (i_1,
\ldots, i_m)\},
$$
by the formula
$$
\langle (\bar{\boldsymbol{\partial}}^J K^\lambda)(\cdot, w_0) ,
(\bar{\boldsymbol{\partial}}^I K^\lambda)(\cdot, w_0) \rangle =
\big ( \boldsymbol{\partial}^I\bar{\boldsymbol{\partial}}^J
K^\lambda \big ) (w_0, w_0).
$$
Now, if $K^\lambda$ were positive definite, for the $\lambda$ we
have picked, then this formula would extend to an inner product on
the linear span of these vectors. The completion of this inner
product space is then a Hilbert space, which we denote by
$\mathcal H^{(\lambda)}$. The reproducing kernel for the Hilbert
space $\mathcal H^{(\lambda)}$ is easily verified to be the
original kernel $K^\lambda$. The Hilbert space $\mathcal
H^{(\lambda)}$ is independent of the choice of $w_0$.

Now, even if $K^\lambda$ is not necessarily positive definite, we
may ask whether this formula defines an inner product on the
$(m+1)$ dimensional space $\mathcal N^{(\lambda)}(w)$ spanned by
the vectors
$$\{K^\lambda(\cdot, w), (\bar{\partial}_1 K^\lambda)(\cdot, w), \ldots , (\bar{\partial}_m K^\lambda)(\cdot, w)\}.$$
An affirmative answer to this question is equivalent to the
positive definiteness of the matrix
$$
\big ( \!\! \big (\, (\partial_i\bar{\partial}_j K^\lambda) (w,
w)\big ) \!\! \big )_{i,j = 0}^m.
$$

Let $\bar{\boldsymbol{\partial}}_m^{\rm t} = \begin{pmatrix}{1,
\partial_1, \ldots, \partial_m}
\end{pmatrix}$ and $\boldsymbol{\partial}_m$ be its conjugate transpose.
Now,
$$\big ( \boldsymbol{\partial}_m \bar{\boldsymbol{\partial}}_m^{\rm t} K^\lambda)(w,w):=
\big ( \!\! \big (\, (\partial_j\bar{\partial}_i K^\lambda)
(w,w)\big ) \!\! \big )_{i,j = 0}^m,\,\, w\in \Omega\subseteq
\mathbb C^m.$$

\begin{thm}\label{locW}
For a fixed but arbitrary $w$ in $\Omega$ and every $\lambda > 0,$  the $(m+1) \times
(m+1)$ matrix $\big ( \boldsymbol{\partial}_m
\bar{\boldsymbol{\partial}}_m^{\rm t} K^\lambda)(w,w)$ is positive
definite.
\end{thm}

\begin{proof}
The proof is by induction on $m$.  For $m=1$ and any positive
$\lambda$, a direct verification, which follows, shows that $$\big
(\boldsymbol{\partial}_1 \bar{\boldsymbol{\partial}}_1^{\rm t}
K^\lambda\big )(w,w):= \left(  \begin{matrix}
K^{\lambda}(w, w) & \partial_{1}K^{\lambda}(w, w)  \\
\bar{\partial}_{1}K^{\lambda}(w, w) &
\partial_{1}\bar{\partial}_{1} K^{\lambda}(w, w)
\end{matrix}\right)$$ is positive.

Since $K^\lambda (w,w) > 0$ for any $\lambda >0$, the verification
that $\big (\boldsymbol{\partial}_1
\bar{\boldsymbol{\partial}}_1^{\rm t} K^\lambda\big )(w,w)$ is
positive definite amounts to showing that $\det \big
(\boldsymbol{\partial}_1 \bar{\boldsymbol{\partial}}_1^{\rm t}
K^\lambda\big )(w,w) > 0$.  An easy computation gives
\begin{eqnarray*}
\det \big (\boldsymbol{\partial}_1
\bar{\boldsymbol{\partial}}_1^{\rm t} K^\lambda\big )(w,w)
 &=& \lambda K^{2 \lambda -2 }(w,w) \big \{ K(w,w) (\bar{\partial_1} \partial_1 K) (w,w) - |\partial_1 K(w,w)|^2 \big \}
 \\
&=& \lambda K^{2 \lambda  }(w,w)\frac{  \|K(\cdot, w)\|^2
\|(\bar{\partial_1}K) (\cdot,w)\|^2 - |\langle K(\cdot, w),
(\bar{\partial_1} K)(\cdot,w)\,\rangle|^2 } {\|K(\cdot, w)\|^{4}},
\end{eqnarray*}
which is clearly positive since $K(\cdot, w)$ and
$(\bar{\partial}_{1} K)(\cdot, w)$ are linearly independent.

Now assume that $\big ( \boldsymbol{\partial}_{m-1}
\bar{\boldsymbol{\partial}}_{m-1}^{\rm t} K^\lambda\big)(w,w)$ is
positive definite. We note that
$$
\big ( \boldsymbol{\partial}_m \bar{\boldsymbol{\partial}}_m^{\rm
t} K^\lambda\big)(w,w) = \left ( \begin{matrix}\big (
\boldsymbol{\partial}_{m-1} \bar{\boldsymbol{\partial}}_{m-1}^{\rm
t} K^\lambda \big)(w,w) & \big (
\partial_{m} \bar{\boldsymbol{\partial}}_{m-1}^{\rm t}
K^\lambda\big)(w,w)
\\ \big (
\boldsymbol{\partial}_{m-1} \bar{\partial}_{m} K^\lambda\big)(w,w)
& ({\partial}_m \bar{\partial}_m
K^\lambda)(w,w)\end{matrix}\right ).
$$
Since $\big ( \boldsymbol{\partial}_{m-1}
\bar{\boldsymbol{\partial}}_{m-1}^{\rm t} K^\lambda\big)(w,w)$ is
positive definite by the induction hypothesis and for
 $\lambda > 0$, we have $$({\partial}_m \bar{\partial}_m
K^\lambda)(w,w) = \lambda K(w,w)^{\lambda-2}\big \{ K(w,w)
(\partial_{m} \bar{\partial}_{m} K)(w,w)
  +  (\lambda-1) |(\bar{\partial}_{m}K)(w,w)|^2 \big \} > 0,$$
it follows that $\big ( \boldsymbol{\partial}_m
\bar{\boldsymbol{\partial}}_m^{\rm t} K^\lambda\big)(w,w)$ is
positive definite if and only if $\det\big(\big (
\boldsymbol{\partial}_m \bar{\boldsymbol{\partial}}_m^{\rm t}
K^\lambda\big)(w,w)\big) > 0$ (cf. \cite{shibu}). To verify this
claim, we note
$$\big ( \boldsymbol{\partial}_m \bar{\boldsymbol{\partial}}_m^{\rm
t} K^\lambda\big)(w,w) = \left ( \begin{smallmatrix}  K^\lambda (w,w) & B(w,w) \\
B(w,w)^* & D(w,w) \end{smallmatrix}\right ),$$ where $D=\big (
\!\! \big (\, (\partial_j\bar{\partial}_i K^\lambda) (w, w)\big )
\!\! \big )_{i,j = 1}^{m}$ and $B=\big(\partial_{1}
K^{\lambda}(w, w), \ldots,
\partial_{m} K^{\lambda}(w, w)\big).$ Recall that (cf. \cite{hal})
$$\det \big ( \boldsymbol{\partial}_m \bar{\boldsymbol{\partial}}_m^{\rm
t} K^\lambda\big)(w,w) = \det \Big (D(w,w) -\frac{B^*(w,w)B(w,w)}
{K^{\lambda}(w, w)} \Big )\det K^{\lambda}(w, w) .$$ Now,
following  \cite[proposition 2.1(second proof)]{shibu}, we
see that
$$D(w,w) -\frac{B^*(w,w)B(w,w)}{K^{\lambda}(w, w)}
 = \lambda K^{2 \lambda - 2} (w,w) \big ( \!\! \big (\, K^{2}(w,w)({\partial_j\,\bar{\partial}_i}\log K) (w, w)\big )
 \!\! \big )_{i,j = 1}^m,$$  which was shown to be a Grammian.
 Thus $D(w,w) -\frac{B^*(w,w)B(w,w)}{K^{\lambda}(w, w)}$ is
 a positive definite matrix and hence its determinant is positive.
\end{proof}
The Theorem we have just proved says that if $E$ is a holomorphic Hermitian vector bundle corresponding to
a Cowen-Douglas operator, then the first order jet bundle $\mathcal J E$ admits a Hermitian structure
(cf. \cite[Section 4.7]{cowen}). It also   prompts the following definition, which is a localization of
the Wallach set to points in $\Omega.$
\begin{defn}
For $\lambda > 0,$ and a positive definite kernel $K$ defined on the domain $\Omega,$ let $K^\lambda(z,w)$ denote the function obtained by polarizing the real analytic function $K(w,w)^\lambda.$ For any two multi indices $\alpha$ and $\beta,$ let $\alpha \leq \beta$ denote the co-lexicographic ordering. Let
$$\mathcal W_{K,\lambda}(w):=\max \Big\{ n\in \mathbb N\Big \vert{ \Big{(}\!\Big(\frac{\big (\partial_z^\alpha\partial_{\bar{w}}^\beta K^\lambda \big )(z,w)}{\alpha
!\beta !}\Big)\!\Big{)}_{0\leq\alpha,\beta\leq\delta}}_{|\,z=w},\, |\delta| = n,\: \mbox{\rm is positive definite}\Big\}.$$
\end{defn}
\begin{rem}
Following Curto and Salinas \cite[Lemma 4.1 and 4.3]{curto},  for a fixed positive definite kernel $K$ on $\Omega,$  
we note that the kernel $K^\lambda,\,\lambda > 0,$ is positive definite if $\mathcal W_{K,\lambda}(w)$ is infinite for some $w\in \Omega.$  A positive definite kernel $K$ is said to be infinitely divisible if 
for every positive real $\lambda,$ the kernel $K^\lambda$ is also positive  definite. It follows that a kernel $K$ is infinitely divisible if and only if
$\mathcal W_{K,\lambda}(w)$ is not finite for some $w$ in $\Omega$ and every $\lambda >0.$  The preceding theorem shows that $\mathcal W_{K,\lambda}(w) \geq 2$ for all $w\in \Omega$ and all $\lambda >0.$
However, Example $1$ of the paper  \cite[pp. 950]{shibu}, shows that $\mathcal W_{K,\lambda}(0),$ $\lambda<\tfrac{1}{2},$ is exactly $2$ for the unit disc. Therefore, Theorem \ref{locW} is sharp.
For a positive definite kernel $K$ defined on $\Omega,$ we believe, it is important to study the behaviour of $\mathcal W_{K,\lambda}(w),$ $w\in \Omega,$ $\lambda >0.$ 
\end{rem}
\section{Bergman Kernel} For any bounded open connected subset
$\Omega$ of $\mathbb C^m$, let $\bSB_\Omega$ denote the Bergman
kernel of $\Omega$. This is the reproducing kernel of the Bergman
space $\mathbb A^2(\Omega)$ consisting of square integrable
holomorphic functions on $\Omega$ with respect to the volume
measure. Consequently, it has a representation of the form
\begin{equation} \label{Bexpanorth}\boldsymbol B_{\Omega}(z,w)=\sum_{k}\varphi_{k}(z)\overline{\varphi_{k}(w)},
\end{equation} 
where $\{\varphi_k\}_{k=0}^\infty$ is any orthonormal basis of
$\mathbb A^2(\Omega)$. This series is uniformly convergent on
compact subsets of $\Omega \times \Omega.$

We now exclusively study the case of the Bergman kernel on the
unit ball $\mathcal D$ (with respect to the usual operator norm)
in the linear space of all $r \times s$ matrices $\mathcal
M_{rs}(\mathbb C)$. The unit ball $\mathcal D$ may be also
described as
$$\mathcal D=\{Z \in\mathcal M_{rs}(\mathbb C):
I-ZZ^* \geq 0\}.$$ The Bergman kernel $\boldsymbol B_{\mathcal D}$ for the domain $\mathcal D$ 
is given by the formula $\boldsymbol B_{\mathcal D}(Z,W)=\det (I-ZW^*)^{-p},$ where
$p=r+s.$ In what follows we give a simple proof of this.

As an immediate consequence of the change of variable formula for
integration, we have the transformation rule for the Bergman
kernel. We provide the straightforward proof.

\begin{lem}\label{lemm:main2}Let $\Omega$ and $\tilde{\Omega}$ be two domains in
$\mathbb C^m$ and $\varphi:\Omega\rightarrow \tilde{\Omega}$ be a
bi-holomorphic map. Then $$\bSB_{\Omega}(z,w)=J_{\mathbb
C}\varphi(z)\overline{J_{\mathbb
C}\varphi(w)}\bSB_{\widetilde{\Omega}}(\varphi(z), \varphi(w))$$
for all $z, w \in \Omega,$ where $J_{\mathbb C}\varphi(w)$ is the
determinant of the derivative $D\varphi(w)$.
\end{lem}
\begin{proof}
Suppose  $\{\tilde{\phi}_n\}$ is an orthonormal basis for $\mathbb
A^{2}(\tilde{\Omega}).$ By the change of variable formula, it follows
easily  that $\{\phi_n\}=\{J_{\mathbb C}\varphi(w)\tilde{\phi}_n\circ
\varphi\}, $ is an orthonormal basis for  $\mathbb
A^{2}(\Omega).$ Hence,
\begin{align*}
\boldsymbol
B_{\Omega}(z,w)=\sum_{n=0}^{\infty}\phi_n(z)\overline{\phi_n(w)}&=\sum_{n=0}^{\infty}J_{\mathbb
C}\varphi(z)(\tilde{\phi}_n\circ \varphi)(z)\overline{J_{\mathbb
C}\varphi(w)(\tilde{\phi}_n\circ \varphi)(w)}\\&=J_{\mathbb
C}\varphi(z)\overline{J_{\mathbb
C}\varphi(w)}\sum_{n=0}^{\infty}\tilde{\phi}_n(\varphi(z))\overline{\tilde{\phi}_n(\varphi(w))}\\&=
J_{\mathbb C}\varphi(z)\overline{J_{\mathbb
C}\varphi(w)}\boldsymbol B_{\widetilde{\Omega}}(\varphi(z),
\varphi(w))
\end{align*} completing our proof.
\end{proof}
If $\Omega$ is a domain in $\mathbb C^m$ and the bi-holomorphic
automorphism group ${\rm Aut}(\Omega)$ is transitive, then we can
determine the Bergman kernel as well as its curvature from its
value at $0.$ A domain with this property is called homogeneous.
For instance, the unit ball $\mathcal D$ in the linear space of
$r\times s$ matrices is homogeneous. If $\Omega$ is homogeneous,
then for any $w\in \Omega$,
 there exists a bi-holomorphic automorphism $\varphi_w$ with the property $\varphi_w(w) = 0$.
 The following Corollary is an immediate consequence of Lemma \ref{lemm:main2}.

\begin{cor}\label{corolo1}For any homogeneous domain $\Omega$ in $\mathbb C^m$, we have
$$\bSB_{\Omega}(w,w)=J_{\mathbb
C}\varphi_{w}(w)\overline{J_{\mathbb
C}\varphi_{w}(w)}\bSB_{\Omega}(0, 0), \,\,  w \in \Omega.$$
\end{cor}

We recall from \cite[Theorem 2]{harris}  that for $Z, W$ in
the matrix ball $\mathcal D$ (of size $r\times s$) and $u\in
\mathbb C^{r\times s}$, we have $$D\varphi_W(Z) \cdot u =  (I-W
W^*)^{\frac{1}{2}}(I-ZW^*)^{-1} u (I-
W^*Z)^{-1}(I-W^*W)^{\frac{1}{2}}.$$ In particular, $D\varphi_W(W)
\cdot u =  (I-WW^*)^{-\frac{1}{2}} u (I-W^* W)^{-\frac{1}{2}}.$
Thus $D\varphi_W(W) = (I-WW^*)^{-\frac{1}{2}} \otimes (I-W^*
W)^{-\frac{1}{2}}.$ We therefore (cf. \cite[exercise 8]{halmos},
\cite{hua}) have
\begin{align*} \det D\varphi_W(W) &=
\big (\det (I-WW^*)^{-\frac{1}{2}} \big )^s \big (\det (I-W^* W)^{-\frac{1}{2}} \big )^r\\
&=\big (\det (I-WW^*)^{-\frac{1}{2}} \big )^{r+s}.
\end{align*}
It then follows that
\begin{align*}
J_{\mathbb C}\varphi_{W}(W)\overline{J_{\mathbb C}\varphi_{W}(W)}
&= \det (I - WW^*)^{-(r+s)},\,\, W\in \mathcal D.
\end{align*}
With a suitable normalization of the volume measure, we may assume
that $\bSB_\mathcal D(0,0) = 1$. With this normalization, we have
\begin{equation} \label{BergKern}
\bSB_{\mathcal D}(W,W) = \det (I - WW^*)^{-(r+s)},\,\, W\in
\mathcal D.
\end{equation}

The Bergman kernel $\bSB_{\Omega},$ where $\Omega=\{(z_1, z_2):
|z_2|
 \leq (1-|z_1|^2)\}
\subset \mathbb C^2$ is known (cf. \cite[Example 6.1.6]{pflug}):
\begin{equation}\label{BergNil}
\bSB_\Omega(z, w)=\frac{3(1-z_1\bar w_1)^2+z_2\bar
w_2}{\{(1-z_1\bar w_1)^2-z_2\bar w_2\}^3},\,\,z,w\in \Omega.
\end{equation}
The domain $\Omega$ is not homogeneous. However, it is a Reinhadt
domain. Consequently, an orthonormal basis consisting of monomials
exists in the Bergman space of this domain. We give a very similar
example below to show that computing the Bergman kernel in a
closed form may not be easy even for very simple Reinhadt domains.
We take $\Omega$ to be the domain
 $$\{(z_1, z_2, z_3):
|z_2|^2
 \leq (1-|z_1|^2)(1-|z_3|^2), 1-|z_3|^2\geq 0\}
\subset \mathbb C^3.$$
\begin{lem}\label{bergman}The Bergman kernel $\bSB_{\Omega}(z,w)$ for the domain $\Omega$ is given by the formula
$$\sum_{p, m, n=0}^{\infty}\frac{m+1}{4\beta(n+1, m+2)\beta(p+1,
m+2)}(z_1\bar{w}_1)^n(z_2\bar{w}_2)^m(z_3\bar{w}_3)^p,$$ where
$\beta(m,n)$ is the Beta function.
\end{lem}
\begin{proof}
Let $\{(z_1)^n(z_2)^m(z_3)^p\}_{n, m, p=1}^{\infty}$ be the
orthonormal basis for the Bergman space $\mathbb A^2(\Omega).$
Now, for the standard measure on $\Omega,$ {\small
\begin{align}\label{eqann:main1}\|(z_1)^n(z_2)^m(z_3)^p\|^2\nonumber&
=\int_{0}^{2\pi}d\theta_1 d\theta_2
d\theta_3\int_{0}^{1}r_{1}^{(2n+1)}dr_1
\int_{0}^{1}r_{3}^{(2p+1)}dr_3\int_{0}^{\sqrt{(1-r_{1}^{2})(1-r_{3}^{2})}}r_{2}^{(2m+1)}dr_2\\\nonumber&=
8\pi^3\int_{0}^{1}r_{1}^{(2n+1)}dr_1
\int_{0}^{1}r_{3}^{(2p+1)}dr_3\frac{(1-r_{1}^{2})^{(m+1)}(1-r_{3}^{2})^{(m+1)}}{2m+2}\\&=\frac{\pi^3}{m+1}
\int_{0}^{1}s_1^{n}(1-s_1)^{(m+1)}ds_1\int_{0}^{1}s_2^{p}(1-s_2)^{(m+1)}ds_2,
\end{align}}where $r_1^{2}=s_1$ and $r_2^{2}=s_2.$
Since $\beta(n, m)=\int_{0}^{1}r^{(n-1)}(1-r)^{(m-1)}dr,$
equation (\ref{eqann:main1}) is equal to{
\begin{align}\|(z_1)^n(z_2)^m(z_3)^p\|^2\nonumber&=\frac{\pi^3}{m+1}\beta(n+1, m+2)\beta(p+1, m+2).
\end{align}}\noindent
From equation (\ref{eqann:main1}), it follows that
$\|1\|^2 = \pi^3\beta(1, 2)\beta(1, 2)=\frac{\pi^3}{4}$.
We normalize the volume measure as discussed above to ensure
{
\begin{align}\|(z_1)^n(z_2)^m(z_3)^p\|^2\nonumber&=\frac{4}{m+1}\beta(n+1, m+2)\beta(p+1,
m+2).\end{align}}\noindent Having computed an orthonormal basis for the
Bergman space, we can complete the computation of the Bergman
kernel using the infinite expansion \eqref{Bexpanorth}.
\end{proof}


The following lemma is a change of variable formula (cf. \cite[1.3.3]{rudin}).
\begin{lem}\label{chain rule}
Suppose $\Omega$ is in $\mathbb C^m, F=(f_1, \ldots, f_n)$ maps
$\Omega$ holomorphically into $\mathbb C^n, g$ maps the range of $F$ into $\mathbb
C,$ and $f_1, \ldots, f_n, g$ are of class $\mathcal C^2.$ If
$$h=g\circ F=g(f_1, \ldots, f_n)$$ then, for $1 \leq i,j\leq m$ and
$z \in \Omega,$ $$\big(\overline{D}_jD_ih
\big)(z)=\sum_{k=1}^{n}\sum_{l=1}^{n}\big(\overline{D}_lD_k g
\big)(w)\overline{D_jf_l}(z)D_if_k(z),$$ where
$\overline{D_jf_l}(z)=(\overline{D}_j\bar{f}_l)(z)$ and $w=(f_1(z), \ldots , f_m(z)).$
\end{lem}
This lemma helps us determine the transformation rule for the Bergman metric as follows (cf. \cite[proposition 1.4.12]{krantz}).

\begin{prop}\label{lemm:main3}
Let $\Omega$ and  $\tilde{\Omega}$ be two domains in $\mathbb C^m$
and $\varphi:\Omega\rightarrow \tilde{\Omega}$ be a bi-holomorphic
map. Then $$\mathcal K_{\bSB_{\Omega}}(w)= {\big (D\varphi\big
)(w)}^{\rm t}\mathcal
K_{\bSB_{\widetilde{\Omega}}}(\varphi(w))\overline{\big
(D\varphi\big )(w)}, w \in \Omega,$$
where $\mathcal K_{\bSB_{\Omega}}(w):=-\big (\!\!\big (\tfrac{\partial^2}{\partial w_i \partial
\bar{w}_j}\log\bSB_{{\Omega}}(w,
w) \big)\!\!\big ).$
\end{prop}
\begin{proof}
For any holomorphic function $\varphi$ defined on $\Omega$, we
have $\frac{\partial}{\partial w_i \partial \bar{w}_j}
\log|J_{\mathbb C}\varphi(w)|^{2}=0.$ Combining this with Lemma
\ref{lemm:main2}, we get
\begin{align*}
\frac{\partial^2}{\partial w_i \partial
\bar{w}_j}\log\bSB_{\widetilde{\Omega}}(\varphi(w),
\varphi(w))&=\frac{\partial^2}{\partial w_i \partial
\bar{w}_j}\log|J_{\mathbb
C}\varphi(w)|^{-2}\bSB_{\Omega}(w,w)\\&=-\frac{\partial^2}{\partial
w_i \partial \bar{w}_j}\log|J_{\mathbb
C}\varphi(w)|^{2}+\frac{\partial^2}{\partial w_i \partial
\bar{w}_j}\bSB_{\Omega}(w,w)\\&=\frac{\partial^2}{\partial w_i
\partial \bar{w}_j}\bSB_{\Omega}(w,w).
\end{align*} Also by Lemma \ref{chain rule} with $g(z)=\log\bSB_{\widetilde{\Omega}}(z, z)$ and $F=f$
we have, $$\frac{\partial^2}{\partial w_i \partial
\bar{w}_j}\log\bSB_{\widetilde{\Omega}}(\varphi(w),
\varphi(w))=\sum_{k, l=1}^{n}\frac{\partial\varphi_{k} }{\partial
w_i}(w)\Big (\frac{\partial^2}{\partial z_k \partial
\bar{z}_l}\log\bSB_{\widetilde{\Omega}}\Big )(\varphi(w),
\varphi(w))\frac{\partial\varphi_{l} }{\partial w_j}(w).$$ Hence
\begin{eqnarray*}
\lefteqn{ \left(\!\!\left ( \Big(\frac{\partial^2}{\partial w_i \partial
\bar{w}_j}\log\bSB_{\widetilde{\Omega}}\Big )(\varphi(w),
\varphi(w)\right)\!\!\right)_{i,j}}\\
&\phantom{~~~~~~~~~~~~~~~~~~~~}=&\left(\!\!\left(\frac{\partial\varphi_{k}
}{\partial w_i}(w)\right)\!\!\right)_{i,k}\left(\!\!\left
( \Big (\frac{\partial^2}{\partial z_k \partial
\bar{z}_l}\log\bSB_{\widetilde{\Omega}} \Big )(\varphi(w),
\varphi(w))\right)\!\!\right)_{k,l}\left(\!\!\left(\frac{\partial\varphi_{l}
}{\partial w_j}(w)\right)\!\!\right)_{l,j} \\
&\phantom{~~~~~~~~~~~~~~~~~~~~}=&{\big (D\varphi\big )(w)}^{\rm
t}\mathcal K_{\bSB_{\widetilde{\Omega}}}(\varphi(w))\overline{\big
(D\varphi\big )(w)}.
\end{eqnarray*}
Therefore we have the desired transformation rule for the Bergman
metric, namely,
$$\mathcal K_{\bSB_{\Omega}}(w)= {\big (D\varphi\big
)(w)}^{\rm t}\mathcal
K_{\bSB_{\widetilde{\Omega}}}(\varphi(w))\overline{\big
(D\varphi\big )(w)},\,\, w \in \Omega.$$
\end{proof}
As a consequence of this transformation rule, a formula for the Bergman
metric at an arbitrary $w$ in $\Omega$ is obtained from its value
at $0$. The proof follows from the transitivity of the
automorphism group.

\begin{cor}\label{transK} For a homogeneous domain $\Omega$, pick a  a  bi-holomorphic  automorphism
$\varphi_w$ of $\Omega$ with $\varphi_w(w) = 0$, $w\in \Omega.$ We
have
$$\mathcal K_{\bSB_{\Omega}}(w)=
\big(D \varphi_{w}(w)\big)^{\rm t}\mathcal
K_{\bSB_{\Omega}}(0)\overline{D\varphi_{w}(w)}$$ for all $ w \in
\Omega.$
\end{cor}
For the matrix ball $\mathcal D$, as is well-known (cf. \cite{FK}), $\bSB_\mathcal
D^\lambda$ is not necessarily positive definite for all $\lambda >
0$. However, as we have pointed out before, the space $\mathcal
N^{(\lambda)}(w)$ has a natural inner product induced by
$\bSB_\mathcal D^\lambda$. Thus we explore properties of
$\bSB_\mathcal D^\lambda$ for all $\lambda >0$. In what follows,
we will repeatedly use the transformation rule for
$\bSB_\Omega^\lambda$ which is an immediate consequence of the
transformation rule for $\bSB_\Omega,$ namely,
\begin{equation}\label{transK^l}
\mathcal K_{\bSB^{\lambda}_{\Omega}}(w)=\lambda\mathcal
K_{\bSB_{\Omega}}(w)=\lambda{D\varphi_{w}(w)}^{\rm t}\mathcal
K_{\bSB_{\Omega}}(0)\overline{D\varphi_{w}(w)}
\end{equation} for $w \in \Omega$ and
$\lambda >0$.

To compute the Bergman metric, we begin with a Lemma on the Taylor
expansion of the determinant. To facilitate its proof, for $Z$  in
$\mathcal M_{rs}(\mathbb C),$ we write  $Z=\left (
\begin{smallmatrix} Z_1 \\ \vdots\\ Z_r
\end{smallmatrix}\right ),$ with $Z_i=\left(z_{i1}, \ldots,
z_{is}\right),$   $i=1 , \ldots, r.$ In this notation,
$$I-ZZ^*=\left (
\begin{smallmatrix} 1-\|Z_1\|^2 & -\langle Z_1, Z_2\rangle &
\cdots
 & -\langle Z_1, Z_r\rangle  \\
 \vdots & \vdots & \vdots & \vdots\\
 -\langle Z_r, Z_1\rangle &  -\langle Z_r, Z_2\rangle & \cdots & 1-\|Z_r\|^2
\end{smallmatrix}\right ),$$ where
$\|Z_i\|^2=\sum_{j=1}^{s}|z_{ij}|^2, \langle Z_i, Z_j\rangle
=\sum_{k=1}^{s}z_{ik}\bar{z}_{jk}.$ Set $X_{ij}=\langle Z_i,
Z_j\rangle, 1\leq i, j \leq r.$

The curvature $\mathcal K_{\boldsymbol B_{\mathcal D}}(0)$ of the
Bergman kernel, which is often called the Bergman metric, is
easily seen to be $p$ times the $rs \times rs$ identity as a
consequence of the following lemma, where $p=r+s.$ The value of the curvature
$\mathcal K_{\boldsymbol B_{\mathcal D}}(W)$ at an arbitrary point
$W$ is then easy to write down using the homogeneity of the unit
ball $\mathcal D$.

\begin{lem}\label{lem:det}
$\det (I-ZZ^*)=1-\sum_{i=1}^{r}\|Z_i\|^2+ P(X),$
where $P(X)=\sum_{|{\ell}|\geq2}p_{\ell}X^{\ell}$ with
$$X^{\ell}:=X_{11}^{\ell_{11}}\ldots X_{1r}^{\ell_{1r}}\ldots
X_{r1}^{\ell_{r1}}\ldots X_{rr}^{\ell_{rr}}.$$
\end{lem}
\begin{proof}
The proof is by induction on $r$.  For $r=1$ we have $\det
(I-ZZ^*)=1-\|Z\|^2.$ Therefore in this case , $P=0$ and we are
done. For $r=2,$ we have
$$\det (I-ZZ^*)=\det\left (
\begin{smallmatrix} 1-\|Z_1\|^2 & -\langle Z_1,
Z_2\rangle \\
 -\langle Z_2, Z_1\rangle  & 1-\|Z_2\|^2
\end{smallmatrix}\right ).$$ For $r=2,$  a
direct verification shows that the $\det (I-ZZ^*)$  is equal to
$1-\sum_{i=1}^{2}\|Z_i\|^2+ P(X),$ where
$P(X)=X_{11}X_{22}-|X_{12}|^2.$ The decomposition
$$I-ZZ^*= \left ( \begin{array}{cccc|c}
1-\|Z_1\|^2 & -\langle Z_1, Z_2\rangle & \cdots
 & -\langle Z_1, Z_{r-1}\rangle &-\langle Z_1,Z_r \rangle \\
 \vdots & \vdots & \vdots & \vdots & \vdots \\
 -\langle Z_{r-1}, Z_1\rangle &  -\langle Z_{r-1}, Z_2\rangle & \cdots & 1-\|Z_{r-1}\|^2 &
  - \langle Z_{r-1} , Z_r \rangle \\ \hline

-\langle Z_r,Z_1 \rangle &  -\langle Z_r,Z_2 \rangle & \cdots &
 - \langle Z_{r} , Z_{r-1} \rangle & 1- \|Z_r\|^2\\
\end{array} \right )$$
is crucial to our induction argument. Let $A_{ij}$, $i,j=1,2$,
denote the blocks in this decomposition.
 By induction hypothesis, we have
$$\det A_{11}=1-\sum_{i=2}^{r}\|Z_i\|^2+ Q(X),$$ where
$Q(X)=\sum_{|\ell|\geq2}q_{\ell}X^{\ell}.$  Since  $\det
(A_{22}-A_{21}A_{11}^{-1}A_{12})$ is a scalar, it follows that
\begin{eqnarray*}
\det (I-ZZ^*) &=& (A_{22}-A_{21}A_{11}^{-1}A_{12})\,\det A_{11}\\
&=& A_{22} \det A_{11} - A_{21}\big (\det A_{11}\big ) A_{11}^{-1} A_{12}\\
&=&  A_{22} \det A_{11} - A_{21}\big ( {\rm Adj} (A_{11}) \big
)A_{12},
\end{eqnarray*}
where, as usual, ${\rm Adj}(A_{11})$ denotes the transpose of the
matrix of co-factors of $A_{11}$. Clearly, $A_{21}\big ( {\rm Adj}
(A_{11}) \big )A_{12}$ is a sum of $(r-1)^2$ terms. Each of these
is of the form $ X_{k 1} a_{j k} X_{1 j}$, where $a_{jk}$ denotes
the $(j,k)$ entry of ${\rm Adj}(A_{11})$. It follows that any one
term in the sum $A_{21}\big ( {\rm Adj} (A_{11}) \big )A_{12}$ is
some constant multiple of $X^{\ell}$ with $|\ell|\geq2.$
Furthermore, $$A_{22} \det A_{11} =
1-\sum_{i=1}^{r}\|Z_i\|^2+\|Z_r\|^2\sum_{i=1}^{r-1}\|Z_i\|^2+Q(X)(1-\|Z_r\|^2).$$
Putting these together, we see that   $$\det
(I-ZZ^*)=1-\sum_{i=1}^{r}\|Z_i\|^2+ P(X),$$ where
$P(X)=X_{rr}\sum_{i=1}^{r-1}X_{ii}+Q(X)(1-X_{rr})-A_{21} \big
({\rm Adj}(A_{11})\big ) A_{12}$ completing the proof.
\end{proof}

 Let $\mathcal K_{\boldsymbol B_{\mathcal D}}(Z)$ be the curvature (sometimes also called the Bergman metric) of
 the Bergman kernel $\boldsymbol B_{\mathcal D}(Z,Z).$
 Set $w_1=z_{11}, \ldots, w_{s}=z_{1s}, \ldots, w_{rs-s+1}=z_{r1},\ldots ,w_{rs}=z_{rs}.$ The formula for
 the Bergman metric given below is due to Koranyi
(cf. \cite{Adam}).

\begin{thm}\label{curva}
$\mathcal K_{\boldsymbol B_{\mathcal D}}(0)=pI,$ where $I$ is the
$rs \times rs$ identity matrix.
\end{thm}
\begin{proof}
Lemma \ref{lem:det} says that
$$\log \boldsymbol B_{\mathcal D}(Z)=-p\log\big(1-\sum_{i=1}^{r}\|Z_i\|^2+ P(X)\big).$$ It now
follows that $\big (\frac{\partial ^{2}}{\partial
w_{i}{\partial}\bar{w}_{j}}\log \boldsymbol B_{\mathcal D}\big )
(0) = 0,$  $i \neq j$. On the other hand, $\big (\frac{\partial
^{2}}{\partial w_{i}{\partial}\bar{w}_{i}}\log \boldsymbol
B_{\mathcal D}\big ) (0)=p,$  $i=1,\ldots , rs.$
\end{proof}

In consequence, for the matrix ball $\mathcal D$, which is a
homogeneous domain in $\mathbb C^{r\times s}$, we record
separately the transformation rule:
\begin{align}\label{transinvtranspose}
\big (\mathcal K_{\bSB_{\mathcal D}}(W)^{\rm t}\big )^{-1} &=
\big(D \varphi_{W}(W)\big)^{-1} \big (\mathcal K_{\bSB_{\mathcal
D}}(0)^{\rm t}\big )^{-1}
 \big (\overline{D\varphi_{W}(W)}^{\rm \,\,t}\big )^{-1}\nonumber\\
&=\frac{1}{p}\big ( \overline{D\varphi_{W}(W)}^{\rm \,\,t}D
\varphi_{W}(W)\big ) ^{-1}, \,\,W\in \mathcal D,
\end{align}
where $p=r+s$.

\section{Curvature inequalities}
\subsection{The Euclidean Ball}
Let $\Omega$ be a homogeneous domain and
$\theta_{w}:\Omega\rightarrow \Omega$ be a bi-holomorphic
automorphism of $\Omega$ with $\theta_{w}(w)=0.$  The linear map
$D\theta_{w}(w): (\mathbb C^m,  C_{\Omega, w}) \to (\mathbb C^m,
C_{\Omega,0})$ is a contraction by definition. Since $\theta_w$ is
invertible, $D\theta_w^{-1}(0): (\mathbb C^m, C_{\Omega,0}) \to
(\mathbb C^m, C_{\Omega, w})$ is also a contraction. However,
since $D\theta_w^{-1}(0) = D\theta_w(w)^{-1}$, it follows that
$D\theta_w(w)$ must be an isometry.  We paraphrase Theorem 5.2 
from \cite{GM}. 
\begin{lem} \label{A(w)} If $\Omega$ is a homogeneous domain, 
then for any $w\in \Omega,$ we have that  $$\|A(w)^{\rm t}
\|_{\ell^2 \rightarrow  C_{\Omega, w}}\leq 1 \mbox{~if and only if~} \|A(0)^{\rm t}\|_{\ell^2 \rightarrow  C_{\Omega, 0}}\leq 1,$$
where $A(w)$ is determined from the equation $\big (-\mathcal K_{\boldsymbol T}(w)^{\rm t}\big )^{-1}
= A(w)^{\rm t}\overline{A(w)}$  as in \eqref{curvform}. 
\end{lem}
\begin{proof}
As before, let $\boldsymbol D_{w}\Omega:=\{D f(w):f\in {\rm
Hol}_w(\Omega, \mathbb D)\}.$ The map $\varphi\mapsto
\varphi\circ\theta_w$ is  injective from ${\rm Hol}_0(\Omega,
\mathbb D)$ onto ${\rm Hol}_w(\Omega, \mathbb D).$ Therefore,
\begin{align*}
\boldsymbol D_{w}\Omega &=\{D(f\circ\theta_{w})(w): f \in {\rm
Hol}_0(\Omega, \mathbb D)\}\\&=\{Df(0)D\theta_{w}(w): f \in {\rm
Hol}_0(\Omega, \mathbb D)\}\\&=\{u  D\theta_w(w): u \in
\boldsymbol D_{0}\Omega\}
\end{align*}
This is another way of saying that $D\theta_w(w)^{\rm tr}:(\mathbb C^m, C_{\Omega,0})^* \to (\mathbb C^m, C_{\Omega,w})^*$ 
is an isometry. Now, we have 
\begin{align*}
\sup\{|v A(w)^{\rm t} x| : v \in \boldsymbol D_{w}\Omega,\, \|x\|_2 \leq 1\}
&=\sup\{|u D\theta_w(w) A(w)^{\rm t} x| : u \in \boldsymbol D_{0}\Omega,\, \|x\|_2 \leq 1\}
\|\\&=\sup\{| u A(0)^{\rm t} x| : u \in \boldsymbol D_{0}\Omega,\, \|x\|_2 \leq 1\},
\end{align*} where $A(0)^{\rm t}:=D\theta_w(w) A(w)^{\rm t}.$ Thus we have shown
\begin{align*} \label{Dw0}
\{ A(w)^{\rm t}: \|A(w)^{\rm t} \|_{\ell^2 \rightarrow C_{\Omega,
w}}\leq 1\}&= \{D\theta_w(w)^{-1} A(0)^{\rm t}: \|A(0)^{\rm t}
\|_{\ell^2 \rightarrow  C_{\Omega, w}}\}\\&=\{D\theta^{-1}_{w}(0) A(0)^{\rm
t}: \|A(0)^{\rm t} \|_{\ell^2 \rightarrow
C_{\Omega, w}}\}.
\end{align*}
The proof is now complete since $D\theta^{-1}_{w}(0)$ is an isometry.
\end{proof}



We  note that if $\|A(w)^{\rm t}\|_{\ell^2 \rightarrow C_{\Omega,
w}}\leq 1,$ then
\begin{align} \|\big
(\mathcal K_{\boldsymbol T} (w)^{\rm t} \big )^{-1}\|_{ C_{\Omega,
w}^*\rightarrow
 C_{\Omega, w}}\nonumber &= \|A(w)^{\rm
t}\overline{A(w)}\|_{ C_{\Omega, w}^* \rightarrow  C_{\Omega,
w}}\\\nonumber &\leq \|A(w)^{\rm t}\|_{\ell^2 \rightarrow
C_{\Omega, w}} \|\overline{A(w)}\|_{ C_{\Omega, w}^* \to \ell^2}\\
& = \|A(w)^{\rm t}\|^2_{\ell^2 \rightarrow C_{\Omega, w}} \leq
1,\end{align} which is the curvature inequality of 
\cite[Theorem 5.2]{GM}. For a homogeneous domain $\Omega$, using
the transformation rules in Corollary \ref{transK} and the
equation \eqref{transinvtranspose}, for the curvature $\mathcal K$
of the Bergman kernel $\bSB_\Omega$, we have
\begin{align}
\|\big (\mathcal K_{\boldsymbol T} (w)^{\rm t} \big )^{-1}\|_{
C_{\Omega, w}^*\rightarrow  C_{\Omega, w}} \nonumber&=\big \|
{\big ( D\theta_w(w)^{\rm t} \mathcal K(0) \overline{D\theta_w(w)}
\big )^{\rm t}}^{-1} \big \|_{ C_{\Omega, w}^*\rightarrow
C_{\Omega, w}}\\\nonumber &=\big \| D\theta_w(w)^{-1}\big
(\mathcal K(0)^{\rm t}\big )^{-1} \overline{
D\theta_w(w)^{-1}}^{\rm t} \|_{ C_{\Omega, w}^*\rightarrow
C_{\Omega, w}}\\\nonumber &=\big \| D\theta_w(w)^{-1} A(0)^{\rm t}
\overline{A(0)} \overline{ D\theta_w(w)^{-1}}^{\rm t}
\|_ { C_{\Omega, w}^*\rightarrow  C_{\Omega, w}}\\
&\leq \big \| D\theta_w(w)^{-1} A(0)^{\rm t}\|_{\ell^2 \to
 C_{\Omega, w}}^2 =  \big \| A(0)^{\rm t}\|_{\ell^2
\to  C_{\Omega, 0}}^2
\end{align}
since $ D\theta_w(w)^{-1}$ is an isometry.  For the Euclidean ball
$\mathbb B:= \mathbb B^n$, the inequality for the curvature is
more explicit. In the following, we set $\mathfrak B(w,w):= \big (
\bSB_\mathbb B(w,w)\big )^{-\frac{1}{n+1}}$. Thus polarizing
$\mathfrak B$, we have $\mathfrak B(z,w) = \big ( 1 - \langle z,
w\rangle)^{-1}$, $z,w \in \mathbb B$.
 The inequality appearing below (cf. \cite{GM}) is a point-wise inequality with respect to
 the usual ordering of Hermitian
 matrices.
\begin{thm}\label{lemm:main4}Let $\theta_w$ be a
bi-holomorphic automorphism of $\mathbb B$ such that
$\theta_w(w)=0.$ If $\rho_{\boldsymbol T}$ is a contractive homomorphism
of $\mathcal O(\mathbb B)$ induced by the localization $N_{\mathbf
T}(w)$, $\boldsymbol T \in \mathrm B_1(\mathbb B),$ then
$$\mathcal K_{\boldsymbol T}(w)\leq -\overline{D\theta_w(w)}^tD\theta_w(w)=
\mathcal K_{\mathfrak B}(w),\,\, w\in \mathbb B$$
\end{thm}

\begin{proof}The equation \eqref{Dw0} combined with the equality $C_{\mathbb B,0}=\|\cdot \|_{\ell^2}$ and the
contractivity of $\rho_{\boldsymbol T}$ implies that $\|D\theta_w(w)
A(w)^t\|_{\ell_2\rightarrow\ell_2}\leq 1.$  Hence
\begin{eqnarray*}
I-D\theta_w(w)A(w)^t\overline{A(w)}\,\overline{D\theta_w(w)}^{\rm
\,\,t}\geq
0&\Leftrightarrow&(D\theta_w(w))^{-1}\big(\overline{D\theta_w(w)}^{\rm
\,\,t}\big)^{-1}-A(w)^t\overline{A(w)}\geq
0\\&\Leftrightarrow&A(w)^t\overline{A(w)}\leq(D\theta_w(w))^{-1}\big(\overline{D\theta_w(w)}^{\rm
\,\,t}\big)^{-1}\\&\Leftrightarrow& \big(- \mathcal K_{\boldsymbol
T}(w)^{\rm t}\big)^{-1}\leq \big ( \overline{D\varphi_{w}(w)}^{\rm
\,\,t}D
\varphi_{w}(w)\big ) ^{-1}.\\
\end{eqnarray*}
Since $-\big(\mathcal K_{\boldsymbol T}(w)^{\rm t}\big)^{-1}$ and
$\big ( \overline{D\theta_w(w)}^{\rm t}D\theta_w(w)\big )^{-1}$
are positive definite matrices, it follows (cf. \cite{rajendra})
that $\mathcal K_{\boldsymbol T}(w)\leq
-\overline{D\theta_w(w)}^{\rm t}D\theta_w(w) =\mathcal
K_{\mathfrak B}(w).$
\end{proof}
This inequality generalizes the curvature inequality obtained in
\cite[Corollary 1.2${}^\prime$]{misra} for the unit disc.
However, assuming that $\mathcal K_{\mathfrak B^{-1}\,K}(w)$ is a
non-negative kernel defined on the ball $\mathbb B$ implies
$(\mathfrak B(w))^{-1}K(w)$ is a non-negative kernel on $\mathbb
B$ (cf. \cite[Theorem 4.1]{shibu}), indeed, it must be infinitely
divisible. This stronger assumption on the curvature amounts to
the factorization of the kernel $K(z,w) = \mathfrak
B(z,w)\tilde{K}(z,w)$ for some positive definite kernel
$\tilde{K}$ on the ball $\mathbb B$ with the property: $\big
(\mathfrak B(z,w)\tilde{K}(z,w)\big )^\lambda$ is non-negative
definite for all $\lambda >0$.


For $\lambda >0$, the polarization of the function
$\bSB(w,w)^\lambda$ defines a positive definite kernel
$\bSB^\lambda(z,w)$ on the ball $\mathbb B$  (cf.
\cite[Proposition 5.5]{arazy}).
We note that $\mathcal K_{\bSB^{\lambda}}(w)\leq \mathcal
K_{\mathfrak B}(w)$ if and only if $\mathcal
K_{\bSB^{\lambda}}(0)\leq \mathcal K_{\mathfrak B}(0)=-I.$ Since
$\mathcal K_{\bSB^{\lambda}}(0)=-\lambda(n+1)I,$ it follows that
$\mathcal K_{\bSB^{\lambda}}(w)\leq \mathcal K_{\mathfrak B}(w)$
if and only if $\lambda \geq \frac{1}{n+1}.$
Thus whenever $\lambda \geq \frac{1}{n+1}$, we have the point-wise
curvature inequality for $\bSB^\lambda(w, w)$. However, since the
operator of multiplication by the co-ordinate functions on the
Hilbert space corresponding to the kernel $\bSB^{\lambda}(w, w),$
is not even a contraction for $\frac{1}{n+1} \leq \lambda <
\frac{n}{n+1},$ the induced homomorphism can't be contractive.  We
therefore conclude that the curvature inequality does not imply
the contractivity of $\rho$ whenever $n>1$. For $n=1$, an example
illustrating this (for the unit disc) was given in 
\cite[page 2]{shibu}. Thus the contractivity of the homomorphism
induced by the commuting tuple of local operators $N_{\mathbf
T}(w),$ for $\boldsymbol T\in \mathrm B_1(\mathbb B),$ does not imply
the contractivity of the homomorphism induced by the commuting
tuple of operators $\boldsymbol T$.

\subsection{The matrix ball}
Recall that $\mathcal N^{(\lambda)}(w)$ is the $m+1$ dimensional space spanned by the vectors
$\bSB_\mathcal D^\lambda(\cdot,w),
\bar{\partial}_1\bSB_\mathcal D^\lambda(\cdot, w), \ldots , \bar{\partial}_m\bSB_\mathcal D^\lambda(\cdot,w).$ On this space, there exists a canonical $m$-tuple of jointly commuting nilpotent operators, namely,
$$N_i^{(\lambda)}(w)\big ( \bar{\partial}_j\bSB_\mathcal D^\lambda(\cdot, w) \big ) =
\begin{cases}
 \bSB_\mathcal D^\lambda(\cdot,w) & \mbox{ if } i=j \\
 0 & \mbox{otherwise}
\end{cases}.
$$
We recall that the positive function $\bSB_\mathcal D^\lambda,
\lambda >0,$ defines an inner product on the finite dimensional
space $\mathcal N^{(\lambda)}(w)$ for all $\lambda >0$ irrespective of whether $\bSB_\mathcal D^\lambda$
is positive definite on the
matrix ball $\mathcal D$ or not. Let $N^{(\lambda)}(w)$ denote the commuting $m$-tuple of matrices
 $(N^{(\lambda)}_1(w) + w_1I, \ldots , N^{(\lambda)}_m(w)+w_mI)$ represented with respect to some orthonormal basis in
 $\mathcal N^{(\lambda)}(w).$ If $\bSB_\mathcal D^\lambda$ happens to be positive definite for some
 $\lambda > 0$ (this is the case if $\lambda$ is a natural number), then ${\boldsymbol N}^{(\lambda)}$ is
 nothing but the restriction of the adjoint of the multiplication operators induced by the
 coordinate functions to the subspace $\mathcal N^{(\lambda)}(w)$ in the Hilbert space determined
 by the positive definite kernel  $\bSB_\mathcal D^\lambda.$  In this section, we exclusively
study the contractivity of the
homomorphism $\rho_{_{\!{\boldsymbol N}^{(\lambda)}(w)}}$ induced by the
commuting $m$-tuples ${\boldsymbol N}^{(\lambda)}(w).$

We set $\mathcal K^{(\lambda)}(w):=\mathcal K_{\bSB_\mathcal D^\lambda}(w),$ $w\in \mathcal D.$
If the homomorphism  $\rho_{_{\!{\boldsymbol N}^{(\lambda)}(w)}}$ is
contractive for some $\lambda >0$, then for this $\lambda$, we
have:$\|\big({\mathcal K^{(\lambda)}}^{\rm t}\big)^{-1}(0)\|\leq
1.$ Like the Euclidian Ball, we study several implications of the
curvature inequality in this case, as well.

\begin{thm}\label{pmat}
For $\lambda >0,$  we have $\|\big({\mathcal K^{(\lambda)}}^{\rm
t}\big)^{-1}(0)\|_{C_{\mathcal D, 0}^* \to
C_{\mathcal D, 0}} = \frac{1}{\lambda p},$ $p=r+s.$
\end{thm}
\begin{proof}
We have shown that $\big (\mathcal K^{\rm t}\big )^{-1}(0)=
\frac{1}{p} I_{rs}.$ Since $C_{\mathcal D, 0}$ is the
operator norm on $(\mathcal M)_{rs}$ and consequently $
C_{\mathcal D, 0}^*$ is the trace norm, it follows that
$\|I_{rs}\|_{ C_{\mathcal D, 0}^* \to  C_{\mathcal
D, 0}} \leq 1.$ This completes the proof.
\end{proof}
The following Theorem provides a necessary condition for the
contractivity of the homomorphism induced by the commuting tuple
of local operators ${\boldsymbol N}^{(\lambda)}(w).$

\begin{thm} \label{themm2}
If the homomorphism  $\rho_{_{\!{\boldsymbol N}^{(\lambda)}(w)}}$ is
contractive, then $\nu \geq 1,$ where $\nu= \lambda p.$
\end{thm}
\begin{proof}
The matrix unit ball $\mathcal D$ is homogeneous. Let
$\theta_{w}$ be the bi-holomorphic automorphism of $\mathcal D$
with $\theta_{w}(w)=0.$ We have seen that $A(w)^{\rm t}=A(0)^{\rm
t}D\theta_w^{-1}(0).$ Since $D\theta_w^{-1}(0)$ is an isometry,
therefore the contractivity of
$\rho_{_{\!{\boldsymbol N}^{(\lambda)}(0)}}$ implies the contractivity
of  $\rho_{_{\!{\boldsymbol N}^{(\lambda)}(w)}},$  $w \in \Omega$ (see
Lemma \ref{A(w)}). The contractivity of
$\rho_{_{\!{\boldsymbol N}^{(\lambda)}(w)}}$ is equivalent to
$\|A(0)^{\rm t}\|_{\ell^2 \rightarrow  C_{\mathcal D,
0}}\leq 1.$ Therefore the contractivity of
$\rho_{_{\!{\boldsymbol N}^{(\lambda)}(w)}},$ for some $w \in \mathcal
D,$ implies $\|\big({\mathcal K^{(\lambda)}}^{\rm
t}\big)^{-\frac{1}{2}}(0)\|_{C_{\mathcal D, 0}^* \to
 C_{\mathcal D, 0}}\leq 1.$ Theorem \ref{pmat} shows that
$\nu \geq 1.$
\end{proof}

If $\lambda >0$ is picked such that $\bSB^\lambda_\mathcal D$ is
positive definite, then Arazy and Zhang (cf. \cite[Proposition
5.5]{arazy}) prove that the homomorphism induced by the commuting
tuple of multiplication operators on the twisted Bergman space
$\mathbb A^{(\lambda)}(\mathcal D)$ is bounded (k-spectral) if and only if $\nu \geq s.$


It follows that if $1\leq \nu<s$, then the homomorphism induced by
the commuting tuple of multiplication operators is not contractive
on the twisted Bergman space $\mathbb A^{(\lambda)}(\mathcal D),$
while the homomorphism  $\rho_{_{\!{\boldsymbol N}^{(\lambda)}(w)}},$  $w
\in \Omega,$ is contractive on the finite dimensional Hilbert
space $\mathcal {\boldsymbol N}^{(\lambda)}(w)$. This is equivalent to the
curvature inequality for $\nu \geq 1.$ However, for $1\leq \nu<s,$
the $rs$-tuple of multiplication operators on the twisted Bergman
space $\mathbb A^{(\lambda)}(\mathcal D)$ is not contractive.
This shows that the curvature inequality is not sufficient for
contractivity of the homomorphism induced by the commuting tuple of
multiplication operators on the twisted Bergman spaces $\mathbb A^{(\lambda)}(\mathcal D),$
when $1\leq \nu<s$ and $n>1.$

We have seen that any commuting tuple of
operators $\boldsymbol T$ in $\mathrm B_1(\mathcal D)$  induces a
homomorphism  $\rho_{_{\!{\boldsymbol N}^{(\lambda)}(w)}}:\mathcal
O(\mathcal D) \rightarrow \mathcal L(\mathbb C^{rs+1}),\, \lambda > 0,$ as in the first paragraph of this subsection. Indeed, what we have said applies equally well to a generalized Bergman kernel, in the language of Curto and Salinas, or to a commuting tuple of operators in the Cowen-Douglas class.
We note that $\rho_{_{\!{\boldsymbol N}^{(\lambda)}(w)}} \otimes I_{rs}: \mathcal
O(\mathcal D) \otimes \mathcal M_{rs} \to \mathcal L(\mathcal
N(w)) \otimes \mathcal M_{rs}$ is given by the formula
\begin{equation}\label{homNmat}
(\rho_{_{\!{\boldsymbol N}^{(\lambda)}(w)}} \otimes I_{rs})(P):=\left(\begin{matrix}P(w) \otimes
I_{rs} & DP(w) \cdot N(w)\\
0 & P(w) \otimes I_{rs} \end{matrix}\right),$$ where $$D P(w)
\cdot N(w)=\partial_1 P(w) \otimes N_1(w)+ \ldots
+\partial_{d}P(w)\otimes N_{rs}(w).
\end{equation}

The contractivity of $\rho_{_{\!{\boldsymbol N}^{(\lambda)}(w)}}\otimes I_{rs},$ as shown
in \cite[Theorem 1.7]{sastry} and \cite[Theorm 4.2]{vern},
is equivalent to the contractivity of the operator
$$\partial_1 P(w) \otimes N_1(w)+ \ldots
+\partial_{d}P(w)\otimes N_{rs}(w).$$ Let
$P_\mathbf A$ be the matrix valued polynomial in $rs$ variables:
$$P_{\mathbf
A}(z)=\sum_{i=1}^{r}\sum_{j=1}^{s}z_{ij}E_{ij},$$ where $E_{ij}$
is the $r \times s$ matrix whose $(i, j)$th entry is $1$ and
other entries are $0$. Let $V=\left(\begin{smallmatrix}V_1\\
\vdots\\V_{rs}\end{smallmatrix}\right)$ be the $rs \times rs$
matrix, where
\begin{eqnarray*}V_1=(v_{11}, 0, \ldots, 0),\ldots ,V_{rs}=(
0, \ldots ,0, \ldots, v_{rs}).
\end{eqnarray*} We compute  the norm of $(\rho_{_{\!{\boldsymbol N}^{(\lambda)}(w)}} \otimes
I_{rs})(P_{\mathbf A}).$
\begin{thm}\label{P_A} For $\rho_{_{\!{\boldsymbol N}^{(\lambda)}(w)}}\otimes I_{rs}$ as above,
we have
$$\|(\rho_{_{\!{\boldsymbol N}^{(\lambda)}(w)}}\otimes I_{rs})(P_{\mathbf
A})\|^2=\max\{\sum_{i=1}^{s}|v_{1i}|^2, \ldots,
\sum_{i=1}^{s}|v_{ri}|^2\}.$$
\end{thm}
\begin{proof}
We have
\begin{eqnarray*}
\|\big(\rho_{_{\!{\boldsymbol N}^{(\lambda)}(w)}} \otimes I_{rs}\big)(P_{\mathbf A})\|^2
&=&\|V_1\otimes
E_{11}+\ldots+V_s\otimes E_{1s}+V_{s+1}\otimes E_{21}+\ldots +V_{rs}\otimes E_{rs}\|^2\\
&=&\left\|\left(\begin{smallmatrix}V_1& \ldots & V_s\\\vdots &
\vdots & \vdots\\V_{rs-s+1}& \ldots
& V_{rs}\end{smallmatrix}\right)\right\|^2 =\left\|\left(\begin{smallmatrix}W_1\\
\vdots\\W_{r}\end{smallmatrix}\right)\right\|^2,
\end{eqnarray*}
where $W_i=\big(V_{is-s+1},\ldots, V_{is}\big).$ It is easy to see
that $W_iW_{j}^*=0$ for $i\neq j.$ Furthermore,
$W_iW_{i}^*=\sum_{j=1}^{s}|v_{ij}|^2.$ Hence we have
$$\|(\rho_{_{\!{\boldsymbol N}^{(\lambda)}(w)}} \otimes I_{rs})(P_{\mathbf A})\|^2=\max\{\sum_{i=1}^{s}|v_{1i}|^2,
\ldots, \sum_{i=1}^{s}|v_{ri}|^2\}$$ completing the proof of the
theorem.
\end{proof}
Even for the small class of homomorphisms
$\rho_{_{\!{\boldsymbol N}^{(\lambda)}(w)}}$ discussed here, finding the cb norm of
$\rho_{_{\!{\boldsymbol N}^{(\lambda)}(w)}}$ is not easy. However, we determine when
$\|(\rho_{_{\!{\boldsymbol N}^{(\lambda)}(w)}}\otimes I_{rs})(P_{\mathbf A})\|^2 \leq
1.$ This gives a necessary condition for the complete
contractivity of $\rho_{_{\!{\boldsymbol N}^{(\lambda)}(w)}}.$
\begin{thm}\label{complete}If $\|(\rho _{{\boldsymbol N}^{(\lambda)}(w)}\otimes I_{rs})(P_{\mathbf
A})\|^2 \leq 1,$ then  $\nu \geq s.$
\end{thm}
\begin{proof}By Theorem \ref{P_A} we have
$$\|(\rho_{_{\!{\boldsymbol N}^{(\lambda)}(w)}} \otimes I_{rs})(P_{\mathbf
A})\|^2=\max\{\sum_{i=1}^{s}|v_{1i}|^2, \ldots,
\sum_{i=1}^{s}|v_{ri}|^2\}.$$ Since $|v_{ij}|^2=\frac{1}{\nu}, 1
\leq i\leq r, 1\leq j\leq s,$ it is immediate that
$\|(\rho_{_{\!{\boldsymbol N}^{(\lambda)}(w)}}(w) \otimes
I_{rs})(P_{\mathbf A})\|^2\leq 1$ implies $\nu \geq s$ completing
the proof of the theorem.
\end{proof}

As a consequence, it follows  that if $1\leq \nu<s$, then the
homomorphism induced by the commuting tuple of the local operators
${\boldsymbol N}^{(\lambda)}(w)$ is not completely contractive.

\subsection{More examples}
We have  discussed the Bergman kernel $\bSB_{\Omega}(w, w)$ for
the domain $\Omega=\{(z_1, z_2): |z_2|
 \leq (1-|z_1|^2)\}
\subset \mathbb C^2.$ The curvature $\mathcal
K_{\bSB_{\Omega}}(w)=\sum_{i, j=1}^{2}T_{ij}(w)dw_i \wedge
d\bar{w}_j$ of the Bergman kernel $\bSB_{\Omega}(w, w)$ is
(cf.\cite[Example 6.2.1]{pflug}):
$$
T_{11}(w)=6\big(\frac{1}{C(w)}-\frac{1}{D(w)}\big)+12|w_1|^2|w_2|^2\big(\frac{1}{C^2(w)}+\frac{1}{D^2(w)}\big),$$
$$T_{12}(w)=\bar{T}_{21}(w)=6w_1\bar{w}_2(1-|w_1|^2)\big(\frac{1}{C^2(w)}+\frac{1}{D^2(w)}\big),$$
$$T_{22}(w)=3(1-|w_1|^2)^2\big(\frac{1}{C^2(w)}+\frac{1}{D^2(w)}\big),$$
where $ C(w):=(1-|w_1|^2)^2-|w_2|^2$ and
$D(w):=3(1-|w_1|^2)^2+|w_2|^2.$ We have  seen that the
polarization $\bSB_{\Omega}^\lambda(z,w)$ of the function
$\bSB_{\Omega}(w,w)^\lambda$
 defines a Hermitian structure for $\mathcal N^{(\lambda)}(w).$ Specializing to $w=0,$  since $-\big (\mathcal K(0)^{\rm t}\big )^{-1}= A(0)^{\rm
 t}\overline{A(0)}, $ we have
 $a_{11}^\lambda(0)=\frac{1}{\sqrt{\lambda T_{11}(0)}}$ and
 $a_{22}^\lambda(0)=\frac{1}{\sqrt{\lambda T_{22}(0)}},$  where
 $(A^{\lambda}(0))^{\rm t}=\left(\begin{smallmatrix}
a_{11}^\lambda(0) & 0\\
0 & a_{22}^\lambda(0)\\
\end{smallmatrix}\right).$
\begin{prop} \label{themm5}The contractivity of the
homomorphism  $\rho_{_{\!{\boldsymbol N}^{(\lambda)}(0)}}$ implies
$16\lambda\geq 5.$
\end{prop}
\begin{proof}
We have $a_{11}^\lambda(0)=\frac{1}{2\sqrt{\lambda}},
a_{12}^\lambda(0)=0, a_{22}^\lambda(0)=\frac{3}{\sqrt{10\lambda}}.
$ Contractivity of the homomorphism $\rho_{_{\!{\boldsymbol N}^{(\lambda)}(0)}}$ is
equivalent to $\|(A^{\lambda}(0))^{\rm t}\|_{\ell^2 \rightarrow
C_{\Omega, 0}}\leq 1 .$ This is equivalent to
$(2(a_{11}^\lambda(0))^2-1)^2\leq (1-(a_{22}^\lambda(0))^2).$
Hence  $16\lambda\geq 5$ completing our proof.
\end{proof}
The bi-holomorphic automorphism group of $\Omega$ is not
transitive. So the contractivity of the homomorphism
$\rho_{_{\!{\boldsymbol N}^{(\lambda)}(0)}}$   does not necessarily imply
the contractivity of the homomorphism
$\rho_{_{\!{\boldsymbol N}^{(\lambda)}(w)}}, w \in \Omega.$  Determining
which of the homomorphism $\rho_{_{\!{\boldsymbol N}^{(\lambda)}(w)}}$ is
contractive appears to be a hard problem.

Let $P_{\mathbf A}:\Omega \rightarrow(\mathcal M_2)_1$ be the
matrix valued polynomial on $\Omega$ defined by $P_{\mathbf
A}(z)=z_1A_1+z_2A_2$ where $A_1= I_2$ and
$A_2=\left(\begin{smallmatrix}
0 & 1\\
    0    & 0
\end{smallmatrix}\right).$ It is natural to ask when $\rho_{_{\!{\boldsymbol N}^{(\lambda)}(w)}}$ is completely contractive.
As before, we only obtain a necessary condition using the
polynomial $P_{\mathbf A}.$

\begin{prop}\label{themm22}
$\|\rho^{(2)}_{{\boldsymbol N}^{(\lambda)}(0)}(P_{\mathbf A})\|\leq 1$ if and
only if $\lambda \geq \frac{11}{20}.$
\end{prop}
\begin{proof}
Suppose that  $\|\rho^{(2)}_{{\boldsymbol N}^{(\lambda)}(0)}(P_{\mathbf
A})\|\leq 1.$  Then we have
$(a_{11}^\lambda(0))^2+(a_{22}^\lambda(0))^2\leq 1.$ Hence
$\lambda \geq \frac{11}{20}.$ The converse verification is also
equally easy.
\end{proof}

We conclude that if $\frac{5}{16}\leq \lambda < \frac{11}{20},$
the homomorphism $\rho_{{\boldsymbol N}^{(\lambda)}}(0)$ is
contractive but not completely contractive. An explicit
description of the set
$$\{\lambda: \|\rho^{(2)}_{{\boldsymbol N}^{(\lambda)}(w)}(P_{\mathbf A})
\|_{\rm op}\leq 1, w \in \Omega\}$$ would certainly  provide
greater insight. However, it appears to be quite intractable, at
least for now.

The formula for the Bergman kernel for the domain
$$\Omega:=\{(z_1, z_2, z_3): |z_2|^2
 \leq (1-|z_1|^2)(1-|z_3|^2), 1-|z_3|^2\geq 0\}
\subset \mathbb C^3.$$ is given in Lemma \ref{bergman}, which gives  
$\bSB_{\Omega}^\lambda(z,0)=1$ and
$\partial_{i}\bSB_{\Omega}^\lambda(z,0)=0$ for $i=1, 2, 3.$ Hence
the desired curvature matrix is of the form $$ \big ( \!\! \big
(\, (\partial_i\bar{\partial}_j \log\bSB_{\Omega}^\lambda) (0,
0)\big ) \!\! \big )_{i,j = 1}^m.$$ Let
$T_{ij}(0)=\partial_i\bar{\partial}_j \log\bSB_{\Omega}^\lambda
(0, 0),$ that is, $\mathcal K_{\bSB_{\Omega}}(0)=\sum_{i,
j=1}^{3}T_{ij}(0)dw_i \wedge d\bar{w}_j.$ An easy computation
shows that $T_{11}(0)=3\lambda=T_{33}(0),
T_{22}(0)=\frac{9\lambda}{2}$ and $T_{ij}(0)=0$ for $i \neq j.$ As
before, we have $a_{11}^\lambda(0)=\frac{1}{\sqrt{T_{11}(0)}},
 a_{22}^\lambda(0)=\frac{1}{\sqrt{T_{22}(0)}}$ and
 $a_{33}^\lambda(0)=\frac{1}{\sqrt{T_{33}(0)}},$ where
$A(0)^{\rm t}=\left(\begin{smallmatrix}
a_{11}^{\lambda}(0) & 0 &0\\
0 & a_{22}^{\lambda}(0) &0 \\
0 & 0& a_{33}^{\lambda}(0)\\
\end{smallmatrix}\right).$

\begin{prop}\label{themm10} The contractivity of the
homomorphism  $\rho_{_{\!{\boldsymbol N}^{(\lambda)}(0)}}$  implies
$\lambda\geq \frac{1}{4}.$
\end{prop}
\begin{proof}
From Lemma (\ref{lemm:main4}) we have
$a_{11}^{\lambda}(0)=\frac{1}{\sqrt{3\lambda}},
a_{12}^{\lambda}(0)=a_{13}^{\lambda}(0)=0,
a_{22}^{\lambda}(0)=\frac{\sqrt{2}}{3\sqrt{\lambda}},
a_{23}^{\lambda}(0)=0$ and
$a_{33}^{\lambda}(0)=\frac{1}{\sqrt{3\lambda}}.$ The contractivity of the homomorphism  $\rho_{_{\!{\boldsymbol N}^{(\lambda)}(0)}}$
is the requirement $\big \| A(0)^{\rm t}\|_{\ell^2 \to
 C_{\Omega, 0}}^2 \leq 1,$ which is equivalent to
$|a_{11}^{\lambda}(0)|^2(1-| a_{33}^{\lambda}(0)|^2)\geq
(|a_{22}^{\lambda}(0)|^2-| a_{33}^{\lambda}(0)|^2).$  Hence we
have $\lambda\geq \frac{1}{4}.$
\end{proof}
For our final example, let $P_{\mathbf A}:\Omega
\rightarrow(\mathcal M_2)_1$ be the matrix valued polynomial
on $\Omega$ defined by $P_{\mathbf A}(z)=z_1A_1+z_2A_2+z_3A_3$
where $A_1=\left(\begin{smallmatrix}
1 & 0\\
    0    & 0
    \end{smallmatrix}\right),
A_2=\left(\begin{smallmatrix}
0 & 1\\
    0    & 0
    \end{smallmatrix}\right), A_3=\left(\begin{smallmatrix}
0 & 0\\
    0    & 1
    \end{smallmatrix}\right).$
\begin{prop}\label{themm9}
$\|\rho^{(2)}_{{\boldsymbol N}^{(\lambda)}(0)}(P_{\mathbf A})\|\leq 1$ if and
only if $\lambda \geq \frac{5}{9}.$
\end{prop}
\begin{proof}
Suppose that  $\|\rho^{(2)}_{{\boldsymbol N}^{(\lambda)}(0)}(P_{\mathbf
A})\|\leq 1.$  Then we have
$$\max\{(a_{11}^\lambda(0))^2+(a_{22}^\lambda(0))^2,
(a_{33}^\lambda(0))^2\} \leq 1.$$ Hence $\lambda \geq
\frac{5}{9}.$ The converse statement is easily verified.
\end{proof}

Thus if $\frac{1}{4}\leq \lambda < \frac{5}{9},$ the homomorphism
$\rho_{_{\!{\boldsymbol N}^{(\lambda)}(0)}}$ is contractive but not
completely contractive.

\subsection*{\bf Acknowledgment}

The authors thank Michel Dristchel and H. Upmeier for
their comments on this work which were very helpful in preparing the final draft. The authors also thank  Dr. Cherian Varughese 
for his suggestions which resulted in many improvements throughout the paper.

\end{document}